\documentclass{article}
\usepackage[T1]{fontenc}
\usepackage{amsmath,amssymb,amsfonts,amsthm,amstext}
\usepackage{nicefrac}
\usepackage[normalem]{ulem}
\usepackage{enumitem}
\usepackage{thmtools}
\usepackage{hyperref}
\hypersetup{
	colorlinks = true,
	linkcolor = blue
}
\usepackage[noabbrev,nameinlink]{cleveref}

\newtheorem{definition}{Definition}[section]
\theoremstyle{plain}
\newtheorem{theorem}[definition]{Theorem}
\newtheorem{proposition}[definition]{Proposition}
\newtheorem{lemma}[definition]{Lemma}
\newtheorem{question}[definition]{Question}
\newtheorem{corollary}[definition]{Corollary}

\theoremstyle{remark}

\newtheorem{example}[definition]{Example}

\theoremstyle{plain}
\newtheorem{mainthm}{Theorem}

\declaretheorem[numbered=no,name=Theorem]{thm*}

\newcommand{\K}{\ensuremath{\mathbb{K}}}
\newcommand{\D}{\ensuremath{\mathbb{D}}}
\newcommand{\R}{\ensuremath{\mathbb{R}}}

\newcommand{\N}{\ensuremath{\mathbb{N}}}

\newcommand{\T}{\ensuremath{\mathbb{T}}}

\newcommand{\eqdef}{\mathrel{\mathop:}=}
\newcommand\restr[2]{{
		\left.\kern-\nulldelimiterspace#1
		\vphantom{\big|}
		\right|_{#2}
}}

\newcommand{\nre}[1]{\| #1 \| }

\DeclareMathOperator{\spn}{span}
\DeclareMathOperator{\codim}{codim}
\DeclareMathOperator{\supp}{supp}

\begin{document}

	\title{Entropy for compact operators and results on entropy and specification.}
	\author{Paulo Lupatini\footnote{lupatini@ime.unicamp.br}, Felipe Carvalho Silva\footnote{felipe.silva@ime.unicamp.br}, and R\'egis Var\~ao\footnote{varao@unicamp.br}}
  \date{}

  \maketitle
    
  \begin{abstract}
    \noindent We investigate the topological entropy of operators. More precisely, in the Banach space setting, we show that compact operators have finite entropy, which depends solely on their point spectrum. Moreover, for operators on \(F\)-spaces, we explore the relationship between the specification property and entropy. In particular, we show that the specification property implies infinite topological entropy, while the operator specification property implies positive entropy. We also show that the invariance principle fails for the class of compact operators.
        
    \noindent\textbf{Keywords.}	Compact Operators. Entropy. Specification.
	\end{abstract}

  \section{Introduction}

    Classical dynamical systems theory on compact spaces have contributed to the recent growth of linear dynamics \cite{antunes22,bayart2009dynamics,grosse2011linear,lopes21} and in this context a widely used approach to detect complex behavior is topological entropy, introduced by Adler, Konheim and McAndrew in \cite{adler1965}. Entropy is often associated with some degree of chaoticity, as it quantifies the exponential rate at which orbits distinguish themselves from one another over time. A higher rate of growth implies richer dynamical behavior, often suggestive of chaos.

    The concept of entropy appears in various scientific disciplines, generally encapsulating the idea of a measure of chaos. Originally formulated in thermodynamics, entropy was later introduced into mathematics by Claude Elwood Shannon, known as the father of information theory. Shortly thereafter, Andrey Kolmogorov and Yakov Sinai developed the notion of entropy in ergodic theory, which later inspired R. L. Adler, A. G. Konheim, and M. H. McAndrew to define topological entropy for dynamical systems.

    A fundamental problem in entropy theory is the classification of Bernoulli shifts based on their entropy, which is closely linked to the Kolmogorov-Bernoulli equivalence~\cite{ORNSTEIN1970337, PV-springer}. It is a striking fact that any two Bernoulli shifts with the same entropy are isomorphic, implying they share the same fundamental dynamical behavior.

    For non-compact spaces, entropy was defined by Bowen for uniformly continuous maps in \cite{bowen1971}, and this is the definition we will adopt (\Cref{def:entropy}). In the special case when the space is compact, Bowen's definition agrees with the classical notion of topological entropy.
    
    To the best of our knowledge, little research has been conducted on the role of entropy in linear dynamics. Some contributions in this direction include \cite{brian2020linear, brian2017spec,yin2018recurrence}. For instance, the authors of \cite{yin2018recurrence} show that, for certain classes of translation operators, infinite topological entropy and hypercyclicity are equivalent. Additionally, \cite{brian2020linear} proves that, within this class of translation operators, nonzero finite entropy is impossible. In \cite{brian2017spec}, the authors claim that the frequently hypercyclic criterion implies infinite entropy\footnote{Although we believe their result to be true, we remain slightly cautious, as their proof of {\cite[Theorem 3.2]{brian2017spec}} relies on the injectivity of a function \(\Phi\), which is not injective in general.}. Furthermore, F. Bayart, V. M\"uller, and A. Peris provided an example of a hypercyclic operator with zero entropy \cite{peris-personal}.

    In contrast, we show that there exists a class of operators with finite entropy, namely compact operators. We prove that the topological entropy of such operators depends solely on their point spectrum.
    
    \begin{mainthm}\label{thmA}
        Let \(T\) be a compact operator acting on a Banach space \(X\). Then the topological entropy is given by:
        \[h_{top}(T) = \sum_{\{\lambda \in \sigma(T) : |\lambda|> 1\}} \log |\lambda|.\]
    \end{mainthm}

    Interestingly, this result closely resembles Bowen's theorem in \cite{bowen1971} regarding the entropy of endomorphisms on finite-dimensional vector spaces. 
    
    This theorem has significant implications for the applicability of classical entropy principles in the context of operator dynamics. In ergodic theory, the variational principle asserts that the topological entropy of a homeomorphism on a compact space equals the supremum of the metric entropies over all invariant probabilities \cite[Theorem 8.6]{walters2000introduction}. However, \Cref{thmA} and \Cref{propositionA1} implies \Cref{cor:fail_VP} which shows that the variational principle fails for the class of compact operators.

    Another noteworthy consequence of \Cref{thmA} is that compact operators ---despite not being hypercyclic ({\cite[Theorem 5.11]{grosse2011linear}}) or chaotic---can have arbitrarily large entropy. Furthermore, if an operators has a sufficiently large point spectrum (i.e. infinitely many linearly independent eigenvectors), an argument similar to that presented in \Cref{ex:entropiashift} can be used to prove infinite entropy. 
    
    Hence, what is the precise relationship between chaotic behavior and entropy in infinite-dimensional dynamical systems?
    
    A more concrete question:
    
    \begin{question}
        Does every frequently hypercyclic operator have infinite topological entropy?
    \end{question}
        
    Our study of entropy naturally led us to the specification property (SP). This was first introduced by Bowen in the context of Axiom A diffeomorphisms in \cite{bowen1970topological}. Since then, it has played a central role in the study of systems exhibiting hyperbolic behavior (such as Anosov diffeomorphisms) by providing insight into their chaotic and ergodic properties \cite{katok.hasselblatt.book}. Several equivalent formulations of the specification property exist; in this work, we adopt the version given in \Cref{def:SP}, following \cite{sigmund1974dynamical}.

    Informally, a dynamical system has the specification property if, given a finite collection of finite orbit segments, there exists a single periodic point that shadows each segment closely, allowing for a uniformly bounded ``transition time'' between segments. This property has been extensively studied in the setting of compact dynamical systems. Notably, it implies strong dynamical consequences such as the density of periodic points, topological mixing, and positive topological entropy; see \cite{sigmund1974dynamical}. This notion, however, does not always adapt well to operator dynamics. For instance, no operator on a Banach space can satisfy the specification property without violating continuity. Nonetheless, it can be meaningfully defined on certain \(F\)-spaces that are not Banach.

    In the setting of linear dynamics, an analogous concept---known as the Operator Specification Property (OSP)---has been considered for operators on non-compact spaces. Our definition, given in \Cref{def:OSP}, is taken from \cite{bartoll2016operators}. In the same work, the authors prove that the OSP implies several important dynamical properties for linear operators, including mixing, frequent hypercyclicity, and chaos. Moreover, they show that any operator satisfying the Frequent Hypercyclicity Criterion also possesses the Operator Specification Property.

    Our main result regarding these two properties is summarized in the following theorem:
    
    \begin{mainthm}\label{thmB}
    Let \(T\) be an operator on a \(F\)-space, then
        \begin{itemize}
            \item[(i)] if \(T\) has the operator specification property, then \(T\) has positive entropy.
            \item[(ii)] if \(T\) has the specification property, then \(T\) has infinite topological entropy.
        \end{itemize}
    \end{mainthm}
        
    In short, this article explores the interplay between topological entropy, the specification property, the operator specification property, and examines their implications for the dynamics of linear operators.

    \vspace{0,3cm}
    \textit{The paper is structured as follows}: In \Cref{sec:prelim}, we provide the necessary definitions and recall several important known results. In \Cref{sec:proofs}, we prove our main results. Specifically, in \Cref{subsec:thmA}, we prove \Cref{thmA} and \Cref{propositionA1}, which states that compact operators only admit ``simple'' invariant probabilities;  we also prove \Cref{cor:fail_VP}, showing that the invariance principle fails for the class of compact operators. In \Cref{subsec:thmB} we prove \Cref{thmB}. Additionally, \Cref{prop:entropy_WSFS} establishes that weighted shifts on sequence spaces have infinite topological entropy, and \Cref{ex:SP-WeightedShift} presents an example of an \(F\)-space in which a weighted shift exhibits the specification property.
    
	\section{Preliminaries}\label{sec:prelim}

	An \(F\)\textit{-space} is a topological vector space \(X\) equipped with a topology induced by a complete translation-invariant metric. An \(F\)\textit{-norm} is an application defined as \(\nre{\cdot} : X\longrightarrow \R\) that satisfies, for all \(x,y\in X\) and every \(\lambda\in\K\), the following conditions:
	\begin{enumerate}
		\item \(\nre{x+y} \leq \nre{x} + \nre{y}\);
		\item \(\nre{\lambda x} \leq \nre{x}\) if \(|\lambda| < 1\);
		\item \(\lim_{\lambda\to 0} \nre{\lambda x} = 0\);
		\item \(\nre{x} = 0\) implies \(x=0\).
	\end{enumerate}

	A canonical example of an \(F\)-norm is the functional \(\nre{x} \eqdef d(x,0)\), where \(d\) is a translation-invariant metric on a vector space \(X\). This framework generalizes Banach spaces, which correspond to the case where \(d\) is induced by a complete norm. In \Cref{subsec:thmB}, we will work with \(F\)-spaces that admit an \textit{unconditional basis}---that is, a sequence \((e_n)_n\) such that every \(x \in X\) has a unique representation
	\[ x = \sum_{n=1}^\infty a_n e_n \]
	converging unconditionally. Important examples include the spaces \(\ell^p(\mathbb{N})\) for \(1 \leq p < \infty\). For a comprehensive treatment of \(F\)-norms and their role in metric linear spaces, we refer to \cite{rolewicz1985metric,kamthan1981sequence}.

	The \textit{spectrum} of an operator \(T\) on a Banach space \(X\) is defined as the set
	\[
    \sigma(T) = \{ \lambda \in \K : T - \lambda I \text{ is not invertible}\}
    \]
	and the \textit{point spectrum} of \(T\) is the set \(\sigma_p(T)\) of the eigenvalues of \(T\). Notably, the spectrum of an operator is always a nonempty and compact. The \textit{spectral radius} of \(T\) is given by
	\[
    r(T) = \sup\{|\lambda|: \lambda \in\sigma(T)\}.
    \]

	The spectral radius theorem states 
    \[
    r(T) = \inf_{n\in\N} \nre{T^n}^{1/n} = \lim_{n\to\infty} \nre{T^n}^{1/n}.
    \]

	The \textit{nullspace} of an operator \(T\) will be denoted by \(N(T)\) and the \textit{range} of \(T\) by \(R(T)\). For \( \lambda \in \K \), we define
    \[
    N_\lambda(T) \eqdef N(T - \lambda I) \quad \textrm{and} \quad R_\lambda(T) \eqdef R(T - \lambda I).
    \]

	An operator acting on a Banach space \(X\) is \textit{compact} if the image of the closed unit ball is relatively compact (i.e., its closure is a compact set). Compact operators form one of the most well-behaved classes of operators on infinite-dimensional spaces, largely due to their simple spectral properties. This fact is illustrated by the following theorem:

	\begin{theorem}[{\cite[Proposition 15.12]{meise1997introduction}}]\label{thm:compactspectral}
		Let \(T\) be a compact operator acting on a Banach space \(X\). Then the following hold:
		\begin{enumerate}[label=\textup{(\roman*)}, align=left, leftmargin=*, itemsep=0.5ex]
			\item \(0\in\sigma(T)\).
			\item Each \(\lambda\in\sigma(T)\setminus\{0\}\) is an eigenvalue of \(T\) and \(\dim N_\lambda(T)\) is finite.
			\item There is a sequence \({(\lambda_n)}_n\) in \(\K\) such that \(\lim \lambda_n = 0\) and \(\sigma(T) = \{0\}\cup\{\lambda_n:n\in\N\}\).
			\item For every \(\lambda\in\K\setminus\{0\}\) we have \(\dim N_\lambda(T) = \codim R_\lambda(T)\); in particular, the Fredholm alternative holds:\[ T-\lambda I \textrm{ is surjective if, and only if, } T-\lambda I \textrm{ is injective.}\]
		\end{enumerate}
	\end{theorem}

	Our analysis will rely crucially on the Riesz Decomposition Theorem, which states:

	\begin{theorem}[{\cite[Theorem B.9]{grosse2011linear}}]\label{thm:rieszdecomp}\index{Riesz Decomposition Theorem}
		Let \(T\) be an operator on the Banach space \(X\). Suppose that the spectrum \(\sigma(T)\) decomposes into two disjoint, nonempty, closed subsets \(\sigma_1\) and \(\sigma_2\), i.e., \(\sigma(T)=\sigma_1\cup\sigma_2\). Then there exist non-trivial \(T\)-invariant closed subspaces \(M_1\) and \(M_2\) of \(X\) such that \(X=M_1\oplus M_2\), and if \(T_i = \restr{T}{M_i}\) for \(i=1,2\) we have \(\sigma(T_1) =\sigma_1\) and \(\sigma(T_2) =\sigma_2\).
	\end{theorem}

	We now formally define the key dynamical properties that will be essential for our results.

	\begin{definition}\label{def:entropy}
		Let \((X,d)\) be a metric space and \(T:X\longrightarrow X\) uniformly continuous application. For any compact subset \(K\) of \(X\), given \(n\in\N\) and \(\varepsilon>0\), a subset \(E\) of \(K\)  is called \((n,\varepsilon)\)-separated with respect to \(T\) if for any distinct \(x,y\in E\), we have \(\max_{0\leq i \leq n-1} d(T^i (x),T^i(y)) > \varepsilon\). Define
		\[s_n(\varepsilon,K,T,d) = \max \{\# E: E \text{ is a } (n,\varepsilon)\text{-separated with respect to } T \},
		\]
		and
		\[
		h(T;K) = \lim_{\varepsilon\to 0}\limsup_{n\to\infty} \frac{1}{n} \log s_n(\varepsilon,K,T,d).
		\]
		The topological entropy of \(T\) is defined as
		\[
		h_{top}(T) = h_d(T) = \sup \{ h(T;K) : K \text{ is a compact subset of } X \}.
		\]
	\end{definition}

	Compact subsets are considered to ensure that \(s_n\) is well-defined. For a more detailed discussion of the concept we refer the reader to one of the books~\cite{walters2000introduction,viana2016foundations}. We emphasize that entropy depends on the metric \(d\). When the choice of metric is clear, we denote topological entropy simply by \(h_{top}\). For future reference, note that if \(d\) and \(d^\prime\) are uniformly equivalent metrics on \(X\) and \(T:X\longrightarrow X\) is uniformly continuous then \(h_d(T) = h_{d^\prime}(T)\). We also mention the fact that functions which do not expand distances have zero topological entropy. The proof of these facts can be found in any of the two previous references. For operators on Banach spaces this means that if \(\nre{T}\leq 1\) then \(h_{top}(T) = 0\).

	The following two results are stated due to their particular importance in our analysis.

	\begin{theorem}[{\cite[Theorem 7.10]{walters2000introduction}}]\label{walters:710}
		\begin{enumerate}[label=\text{\roman*}.,align=left]
			\item[]
			\item If \((X,d)\) is a metric space and \(T:X\longrightarrow X\) uniformly continuous application and \(m>0\) then \(h_d(T^m) = m\cdot h_d(T)\).
			\item Let \((X_i,d_i)\), \(i=1,2\), be metric spaces, and let \(T_i:X_i\longrightarrow X_i\) be uniformly continuous. Define a metric \(d\) on \(X_1\times X_2\) by \[d((x_1,x_2),(y_1,y_2)) = \max \{ d_1(x_1,y_1),d_2(x_2,y_2)\}.\]
            Then \(T_1\times T_2\) is uniformly continuous and \(h_d(T_1\times T_2) \leq h_{d_1}(T_1) + h_{d_2}(T_2)\). If either \(X_1\) or \(X_2\) is compact then \(h_d(T_1\times T_2) = h_{d_1}(T_1) + h_{d_2}(T_2)\).
		\end{enumerate}
	\end{theorem}

	\begin{theorem}[{\cite[Theorem 8.14]{walters2000introduction}}]\label{walters:teo814}
		Let \(V\) be a \(p\)-dimensional vector space and \(A:V\longrightarrow V\) be linear and \(\rho\) a metric induced by a norm on \(V\). Then
		\[ h_\rho(A) = \sum_{\{i : |\lambda_i|>1\}} \log |\lambda_i|, \]
		where \(\lambda_1,\dots,\lambda_p\) are the eigenvalues of \(A\).
	\end{theorem}

    \begin{definition}\label{def:SP}
		Let \((X,d)\) be a metric space and \(f:X \longrightarrow X\) a continuous function. A function \(f\) is said to have the specification property if, for every \(\varepsilon>0\), there exists \(N\in\N\) such that for any \(s \in \N\), any points \(y_{1} , \ldots , y_{s} \in X\), and any integers \(0 \leq a_{1} \leq b_{1} < a_{2} \leq b_{2} < \cdots < a_{s} \leq b_{s}\) which satisfy \(a_{i+1} - b_{i} \geq N\) for \(i=1,\ldots,s-1\), there exists a point \(x \in X\) which is fixed by \(f^{N+b_s}\) and, for each \(i=1,\ldots,s\) and all integers \(n\) with \(a_i \leq n \leq b_i\), we have \(d(f^n(x),f^n(y_i))<\varepsilon\).
	\end{definition}

	It can be easily verified that any linear transformation on a Banach space satisfying \Cref{def:SP} is an unbounded operator. The following definition provides a natural extension of the specification property to the operator setting.

	\begin{definition}\label{def:OSP}
		An operator \(T\) on a separable \(F\)-space \(X\) is said to have the operator specification property if there exists an increasing sequence \({(K_n)}_n\) of \(T\)-invariant sets with \(0\in K_1\) and \[ \overline{\bigcup_{n=1}^{\infty} K_n} = X\] such that for every \(n\in\N\) the map \(\restr{T}{K_n}\) has the specification property.
	\end{definition}

	Observe that both the specification property and the operator specification property imply density of periodic points.

	An important class of operators in our analysis consists of weighted shifts on \(F\)-spaces. Let us recall their precise definition: given a sequence \(w = (w_n)_{n\in\mathbb{N}}\) of nonzero scalars, the \textit{backward weighted shift} \(B_w\) on an \(F\)-space with unconditional basis \((e_n)_{n \in \mathbb{N}}\) is defined by
	\[
	B_w e_n = w_n e_{n-1} \quad \text{(with the convention } e_0 \eqdef 0\text{)}.
	\]
	The corresponding \textit{forward weighted shift} operator \(F_w\) is given by
	\[
	F_w e_n = w_n e_{n+1}.
	\]
	These operators play a fundamental role in the study of linear dynamics on sequence spaces.

	\section{Proof of the results}\label{sec:proofs}

	In this section we prove our main results.

	\subsection{The Proof of Theorem A}\label{subsec:thmA}

	We begin this subsection by establishing several lemmas needed for the proof of \Cref{thmA}, followed by an analysis of entropy calculations in specific examples. Throughout this section, we restrict our attention to operators acting on complex Banach spaces.

	\begin{lemma}\label{lemma:entropycompact}
		Let \(T\) be an operator acting on a Banach space \(X=M_1\oplus M_2\) where \(M_1\) and \(M_2\) are \(T\)-invariant subsets of \(X\). If \(d\) denotes the metric induced by the norm of \(X\) then \(h_d(T)\leq h_d(\restr{T}{M_1}) + h_d(\restr{T}{M_2})\).
	\end{lemma}

	\begin{proof}
		For \(i=1\) or \(2\) denote the operator \(\restr{T}{M_i}\) by \(T_i\), and consider  \(T_1\times T_2 (x,y) = (T_1 x, T_2y)\). The application \(\Phi: M_1\times M_2 \longrightarrow M_1\oplus M_2\) defined by \(\Phi(x,y) = x+y\) is an algebraic isomorphism. Consider \(\rho\) a metric on \( M_1\times M_2 \) given by \(\rho((x_1,x_2),(y_1,y_2)) = \nre{x_1-y_1} + \nre{x_2-y_2}\). Note that \(\Phi\) is continuous and, by the open map theorem, a homeomorphism.

		Given \(n\in\N\) and \(\varepsilon>0\), let \(K=\varPhi(L)\) be a compact subset of \(X\). We prove that if \(E\subset K\) is a \((n,\varepsilon)\)-separated set with respect to \(T\), then \(F = \Phi^{-1}(E) \subset L\) is a \((n,\varepsilon)\)-separated set with respect to  \(T_1 \times T_2\). Given distinct points \(x=(x_1,x_2),y=(y_1,y_2)\in F\) for \(i=0,1,\dots,n-1\),
		\begin{align*}
			\rho({(T_1 \times T_2)}^i(x),{(T_1 \times T_2)}^i(y)) & = \rho( (T_1^i(x_1),T_2^i(x_2)),(T_1^i(y_1),T_2^i(y_2))) \\
			& = \nre{T_1^i(x_1) - T_1^i(y_1)} + \nre{T_2^i(x_2)-T_2^i(y_2)} \\
			& \geq \nre{T^i(x_1+x_2)-T^i(y_1+y_2)} >\varepsilon.
		\end{align*}

		Hence any \(E\subset K\) \((n,\varepsilon)\)-separated set with respect to \(T\) induces an \(F\subset L\) \((n,\varepsilon)\)-separated set with respect to \(T_1\times T_2\) of the same cardinality. This yields 
        \[
        s_n(\varepsilon,K,T,d) \leq s_n(\varepsilon,L, T_1\times T_2,\rho).
        \]
		Applying limits on both sides of the inequality (according to \Cref{def:entropy}) yields 
        \[
        h_d(T;K) \leq h_\rho(T_1\times T_2 ; L).
        \] 
        As \(\Phi\) is a homeomorphism, taking the supremum over compact sets, the previous inequality implies that \(h_d(T)\leq h_\rho(T_1\times T_2)\). As \(\rho\) is uniformly equivalent to the metric \(\rho^\prime\) defined by \(\rho^\prime((x_1,x_2),(y_1,y_2)) =  \max\{ \nre{x_1-y_1}, \nre{x_2-y_2}\}\), \Cref{walters:710} implies the result.
	\end{proof}

	\begin{lemma}\label{lemma:spectrumcontraction}
		If \(\sigma(T)\subset \D\), then there exists \(n\in\N\) such that \(T^n\) is a contraction. In particular, \(h_{top}(T) = 0\).
	\end{lemma}
	\begin{proof}
		By the spectral radius formula and the compactness of \(\sigma(T)\), it follows that \(r(T) < 1\). Choose \(\varepsilon>0\) so that \(r(T) + \varepsilon < 1 - \varepsilon < 1\). Then, there exists a \(N\in\N\) such that for all \(n\geq N\), the following holds:
		\[
		r(T) - \varepsilon < \nre{T^n}^{1/n} < r(T) + \varepsilon < 1 - \varepsilon.
		\]
		Therefore, there is a constant \(M>0\) such that for all \(x \in X\) and all \(n\in\N\)
		\[
		\nre{T^n x}\leq M{( 1 - \varepsilon)}^n\nre{x}.
		\]
		For sufficiently large \(n\), \(\nre{T^n} < 1\). Thus \(T^n\) is a contraction. Since a contraction has zero entropy, it follows from \Cref{walters:710} that
		\[h_{top}(T) = \frac{1}{n}h_{top}(T^n) = 0.\]
	\end{proof}
	
	\begin{lemma}\label{lemma:entropysub}
		Let \(T:X\longrightarrow X\) be an operator. If \(V\) is a \(T\)-invariant subspace of \(X\) then \(h_{top}(T) \geq h_{top}(\restr{T}{V})\).
	\end{lemma}
	\begin{proof}
		In \(V\) we consider the induced metric from \(X\). Given \(\varepsilon>0\), \(n\in\N\) and \(K\) a compact subset of \(V\). Since the induced metric coincides with the original one in \(V\), a \(E\subset K\) is \((n,\varepsilon)\)-separeted set with respect to \(\restr{T}{V}\) if and only if it is also a \((n,\varepsilon)\)-separeted set with respect to \(T\), then, \(s_n(E,K,\restr{T}{V},d) = s_n(E,K,T,d)\). Hence \(h(T_V;K) = h(T;K)\). Consequently,
		\[
		h_{top}(T_V) = \sup_{K\subset V} h(T_V;K) \leq  \sup_{K\subset X} h(T;K) = h_{top}(T).
		\]
	\end{proof}

	As a consequence of \Cref{lemma:entropysub}, we obtain that if an operator has an eigenvalue of modulus greater than one, then it has positive entropy.

	\begin{proof}[Proof of \Cref{thmA}]
		Split the spectrum into the compact disjoint sets \(\sigma_1 = \{\lambda\in\sigma(T) : |\lambda| \leq 1\}\) and \(\sigma_2 = \{\lambda\in\sigma(T): |\lambda| > 1 \}\). If \(\sigma_2\) is empty, \Cref{lemma:spectrumcontraction} implies \(h_d(T) = 0\) and we are done. Suppose that \(\sigma_2\) is not empty. By Riesz's Theorem \ref{thm:rieszdecomp} we can decompose \(X\) into a direct sum of \(T\)-invariant subspaces \(M_1\) and \(M_2\). Denote the operators \(\restr{T}{M_i}\) by \(T_i\), by \Cref{lemma:entropycompact} we have \[ \max\{h_d(T_1),h_d(T_2)\} \leq h_d(T) \leq h_d(T_1) + h_d(T_2). \]
		Again, \Cref{lemma:spectrumcontraction} implies \(h_d(T_1) = 0\). Therefore, \(h_d(T) = h_d(T_2)\). Clearly, \(T_2\) is a compact operator and \(0\notin \sigma(T_2) = \sigma_2\). Hence \(M_2\) is finite-dimensional. By \Cref{walters:teo814} we have
		\[
		h_{top}(T) = h_d(T) = \sum_{\{i : |\lambda_i|>1\}} \log |\lambda_i|
		\]
	\end{proof}

	\begin{example}
		If \(T\) is a hyperbolic operator on a Banach space \(X\), meaning that \(\sigma(T)\cap\T =\varnothing\), it is well known we can find a splitting \((T,X)=(T_s\oplus T_u, X_s\oplus X_u)\) and a norm \(\nre{\cdot}\) on \(X\) such that \(\nre{T_s}<1\). The same argument used in \Cref{thmA} can be used to prove that \(h_{top}(T) = h_{top}(T_u)\).
	\end{example}

	In the following example, similarly to \Cref{prop:entropy_WSFS}, we compute the topological entropy of a weighted shift on a Banach space.

	\begin{example}\label{ex:entropiashift}
		For \(|\alpha| > 1\), the Rolewicz operator \(T=\alpha B:X\longrightarrow X\), where \(X = \ell^p(\N)\) and \( 1\leq p < \infty\) has infinite topological entropy. The sequence \(x_{\lambda} = {( \lambda^n / \alpha^n )}_n \in\ell^p(\N)\) is a scalar multiple of an eigenvector if and only if \(|\lambda/\alpha|<1\). Clearly, for \(|\lambda|< |\alpha|\) we have \(Tx_\lambda = \lambda x_\lambda\), which implies \(\alpha\D \subset \sigma_p(T)\). We can, in fact, show equality. By the spectral radius formula and the compactness of the spectrum, we have \(\sigma(T) = \overline{\alpha\D}\). Thus, it suffices to show that no complex number \(\lambda\) with \(|\lambda| = |\alpha| \) belongs to \(\sigma_p(T)\). Suppose that there is such complex number, meaning there is a nonzero \(y\) such that \(Ty=\lambda y\). Then \(\alpha y_{n+1} = \lambda y_n\) for every \(n\in\N\), hence \( y = y_1{(\lambda ^n/ \alpha^n)}_n\). A simple calculation shows \(y\notin X\). Which implies \(\alpha\D = \sigma_p(T)\).

		Choose \(1<r<|\alpha|\) and let \(\lambda_1,\dots,\lambda_n\) distinct points in the circle \(r \T \subset \sigma_p(T)\). Consider \(V = \spn \{ x_{\lambda_i};i=1,\dots,n \}\). Note \(V\) is a \(n\)-dimensional \(T\)-invariant subspace. Clearly \(h_{top}(\restr{T}{V}) = n \log r\). Since \(n\) is arbitrary, it follows that \(h_{top}(T) = \infty\).
	\end{example}

	For the next result, recall the \textit{support} of a Borel measure \(\mu\) on a topological vector space \(X\) is defined as
	\[\supp \mu = \{x \in X : \forall V\in \mathcal{V}(x), \mu(V)>0\}.\]
	where \(\mathcal{V}(x)\) is a local basis of open neighborhoods of \(x\). If \(A\) is a linear operator on a finite-dimensional normed vector space, its \textit{mini-norm} is given by
	\[ m(A)=\inf \{\nre{Ax}:\nre{x}=1\}.\]

	\begin{lemma}\label{lemma:mininorm}
		If \(A\) is an invertible operator on a finite-dimensional vector space, then \(m(A) = \nre{A^{-1}}^{-1}\).
	\end{lemma}

	The lemma above is a known result from linear algebra which we state for future reference. 

  The previous theorem highlights the importance of spectral decomposition, particularly the separation of the spectrum according to eigenvalue moduli. Through iterated applications of the Riesz theorem, we obtain a decomposition into invariant subspaces \(X = X_u \oplus X_c \oplus X_s\) with the following properties:
  The restrictions \(T_u = \restr{T}{X_u}\), \(T_c = \restr{T}{X_c}\), and \(T_s = \restr{T}{X_s}\) satisfy:
    \begin{align*}
        \sigma(T_u) &= \{\lambda \in \sigma(T) : |\lambda| > 1\}, \\
        \sigma(T_c) &= \{\lambda \in \sigma(T) : |\lambda| = 1\}, \\
        \sigma(T_s) &= \{\lambda \in \sigma(T) : |\lambda| < 1\}.
    \end{align*}
  
		The nomenclature reflects partial hyperbolic dynamics: \(X_s\) (stable), \(X_u\) (unstable), and \(X_c\) (central) correspond to the contraction, expansion, and neutral behavior of \(T\) respectively.

	\begin{proposition}\label{propositionA1}
		For any compact operator \(T\) on a Banach space \(X\), the support of every \(T\)-invariant probability \(\mu\) is contained in \(X_c\).
	\end{proposition}
  
  \begin{proof}
    Using the above decomposition, we will establish that \(\supp\mu \subseteq X_c\) for any \(T\)-invariant probability \(\mu\) by showing that for every \(x = x_u + x_c + x_s\) not in \(X_c\) there is \(V\in\mathcal{V}(x)\) with \(\mu(V) = 0\). 
		
		We will split the proof in two steps. First \(x_u\ne 0\) and second \(x_u =0\).

		\textit{Step 1}. Fix \(x = x_u + x_c + x_s\) with \(x_u\ne 0\). Let \(r = \nre{x_u}\). Consider the set \[W = \{ y \in X: d(y,X_c\oplus X_s)< r+\delta\}\] for some \(\delta >0\) to be specified latter. Clearly, \(W\) is an open set containing \(x\). Note that
		\begin{align*}
			d(y,X_c\oplus X_s) & = \inf \{ \nre{y-z} : z\in X_c\oplus X_s\} \\
			& = \inf \{ \nre{y_u + (y_c+y_s - z)} : z\in X_c\oplus X_s\} \\
			& = \inf \{ \nre{y_u  - w }: w\in X_c\oplus X_s \} \\
			& =  d(y_u,X_c\oplus X_s)
		\end{align*}
		hence \(y\in W\) if and only if \(y_u \in W\). Let \(m = m(T_u)\) be the mini-norm of \(T_u\). By the \Cref{lemma:mininorm}, we have that \(m>1\), furthermore for every \(y\in X\),
		\begin{equation}
			\label{eq:nome}
			m \leq \bigg\lVert T_u\bigg(\frac{y_u}{\nre{y_u}}\bigg)\bigg\rVert \implies \nre{m y_u} \leq \nre{T_u y_u}.
		\end{equation}
		Consequently, if \(y\in T^{-1} W\) then \(T_u y_u \in W\). \Cref{eq:nome} implies \(my\in W\). That is, \(T^{-1} W \subset \nicefrac{1}{m} W\).

		Since \(m>1\) the above inclusion yields \(T^{-1}(W)\subset \nicefrac{1}{m} W \subset W\). By the invariance of the measure \[\mu(W) = \mu(T^{-1} W) \leq \mu(\nicefrac{1}{m} W) \leq \mu(W),\]
		which gives \(\mu(W) = \mu(\nicefrac{1}{m}W)\). The intermediate value theorem allows us to choose \(\delta>0\) such that \[1<\frac{r+\delta}{r-\delta}<m.\] 
		This ensures \(x \in W\setminus\nicefrac{1}{m}\overline{W}\eqdef V\), leading to \(\mu(V) = 0\). Thus, we conclude \(x\notin \supp\mu\).

		\textit{Step 2}. Fix \(x = x_c+x_s\). Let \(2r=d(x,X_u\oplus X_c)\). Consider 
		\[W = \{ y \in X: d(y,X_u\oplus X_x)< r\}.\]  
		By \Cref{lemma:spectrumcontraction}, \(T_s^k\) is a contraction for sufficiently large \(k\in\N\). Hence, for sufficiently large \(m\), we have \(T_s^{km} x_s \in W\). Similarly to the previous step, \(y\in W\) if and only if \(y_s\in W\). 
		Those observations imply that if \(n=km\), we have \(x \in T^{-n}(W)\). Since \(T^{-n}(W)\) is an open set, there is a \(V\in\mathcal{V}(x)\) such that \(V\subset T^{-n}(W)\), \(V\cap W = \varnothing\).

		Note that \(W\subset T^{-n}(W)\). Indeed, for any \(y\in W\), \(\nre{T_s^n y_s} \leq \nre{y_s}\) which implies \(T_s^n y_s \in W\) and thus \(T^n y\in W\).

		By the above observations, \[\mu(V) \leq \mu(T^{-n}(W)\setminus W) = 0,\] which gives \(\mu(V) = 0\). This proves \(x\notin\supp\mu\).

	Hence, 	$\supp\mu \subset X_c$.

	\end{proof}

    \begin{corollary}\label{cor:fail_VP}
    The variational principles fails for compact operator with positive topological entropy.
    \end{corollary}
    \begin{proof}
    Proposition \ref{propositionA1} above implies that if $\mu$ is a $T$-invariant probability for the compact operator $T$, then $\supp \mu \subset X_c$. But $T_c$ is an isometry ($T_c$ is actually a finite direct sum of rotations on $X_c$), hence it is known that isometries have zero metric entropy, therefore the supreme of the metric entropy of invariant probabilities for $T$ is zero, which is different from the topological entropy of $T$ which is positive. Therefore the variational principle fails for $T$.
    \end{proof}

  \subsection{The proof of Theorem B}\label{subsec:thmB}

	We begin this subsection with the proof of \Cref{thmB}. Later, we prove in \Cref{prop:entropy_WSFS} that weighted shifts on sequence spaces have infinite topological entropy. Finally in \Cref{ex:SP-WeightedShift}, we provide an example of an operator on an \(F\)-space with the specification property.

	\begin{proof}[Proof of \Cref{thmB}]\mbox{}\\*
		\textit{Proof of i}. Suppose \(T\) has the operator specification property. Consider the sequence of sets \({(K_n)}_n\) as defined in \Cref{def:OSP}. We can then choose one nonzero periodic point \(x_1\) in some \(K_n\). Suppose this point is \(k\)-periodic. Note that \(T^k\) also satisfies the operator specification property for the same increasing sequence of sets \({(K_n)}_n\) (see {\cite[Proposition 8]{bartoll2016operators}}). Hence \(\restr{T^k}{K_n}\) has the specification property.

		Let \( \delta = 3^{-1}d(x_{1},0)\). By \Cref{def:SP}, there exists an integer \( N = N(\delta) \) such that for any \(n\)-tuple \( y= (y_{1},\ldots, y_{n})\in {\{0,x_1 \}}^n\), there exists a point \(x_{y} \in X\) satisfying
		\[
		d(T^{kj}x_{y} , y_{i}) < \delta \text{, for } j=i(N+1) \text{ and } i=0, \ldots, n-1.
		\]
		Hence the set \({\{ x_{y} \mid y \in \{0, x_1 \}}^{n} \}\) contains \(2^{n}\) elements and it is \(((n-1)(N+1) , \varepsilon)\)-separated, for the compact set \( E = \{0,x_1\}\cup {\{ x_y\mid y \in \{0, x_1 \}}^{n} \}\), for every \(0 < \varepsilon < \delta\) and \(n \in \N\). Hence,
		\[
		h_{top}(T^k) \geq h(T^k,E) \geq \limsup _{n\rightarrow \infty} \frac{n\log(2)}{(n-1)(N+1)} = \frac{\log(2)}{N+1} > 0.
		\]
		This implies \[h_{top}(T) = \frac{h_{top}(T^k)}{k}  > 0.\]

		\textit{Proof of ii}. Now suppose \(T\) has the specification property. Consider a \(k\)-periodic point \(x_{1}\) such that \(x_1 \neq 0\). Let \(\delta = \frac{d(x_{1} , 0)}{3}\) and let \(x_{n} = n \cdot x_{1}\) and \(L_{n} = \{ \alpha x_{1} \mid 0 \leq \alpha \leq n \}\).
		By the specification property, for any fixed \(m\in\N\), there exists \(N = N(\delta) \in\N\) such that for every \(n\)-tuple \(y = (y_{1},\ldots, y_{n}) \in {\{x_{1},\ldots, x_{m}\}}^n\), there exists a \(x_{y} \in X\) satisfying
		\[d(T^{kj}x_{y} , y_{i}) < \delta \text{, for } j=i(N+1) \text{ and } i=0,\ldots, n-1.\]
		Consequently, the set \({\{ x_{y} \mid y \in \{ x_{1},\ldots, x_{m} \}}^{n}\}\) has \(m^{n}\) elements and it is \(((n-1)(N+1) , \varepsilon)\)-separated for every \(0 < \varepsilon < \delta\) and \(n \in \N\), hence
		\[
		h_{top}(T^k) \geq \limsup_{n\rightarrow \infty} \frac{n\log(m)}{(n-1)(N+1)} = \frac{\log(m)}{N+1}
		\]
		Since \(m\) is arbitrary, we conclude that \(h_{top}(T^k) = \infty\), and hence, \(h_{top}(T) = \infty\).
	\end{proof}

	The following proposition follows closely {\cite[Theorem 2]{bartoll2014cantor}}. 
         
	\begin{proposition}\label{prop:entropy_WSFS}
		Let \(X\) be a \(F\)-space with an unconditional basis \(( e_{n})_n\). Suppose the sequence \(w = (w_n)_n\) satisfies
		\[ \sum_{n=1}^\infty {{\bigg(\prod_{i=1}^n w_n \bigg)}}^{-1} e_n \in X,\]
		Then the topological entropy of the weighted shift \(B_w\) is infinite, i.e., \[h_{top}B_w = \infty.\]
	\end{proposition}
	\begin{proof}
		For a given \(N\in\N\), define the map \(\Phi_{N} : {\{ 0, \ldots, N-1 \}}^{\N} \rightarrow X\) by
		\[ \Phi_{N} [{(x_{n})}_{n}] = \sum_{n=1}^\infty \left( \prod_{i=1}^{n} w_{i}^{-1} \right)x_n e_n .\]
		Clearly, when \({\{0, \ldots, N-1\}}^{\N}\) is equipped with the product topology, \(\Phi_N\) is continuous, injective, and satisfies \(\Phi_{N} \circ \sigma_{N} = B_{w} \circ \Phi_{N}\), where \(\sigma_N\) is the Bernoulli shift on \(N\) symbols. In particular, \(\Phi_{N}\) is a homeomorphism onto its image. Consequently if \(K_{N} = \Phi_{N}({\{ 0, \ldots, N-1 \}}^{\N})\), \[h_{top}(B_w) \geq h_{top}(\restr{B_{w}}{K_{N}}) = h_{top}(\sigma_{N}).\]
		Since it is well known that \(h_{top}(\sigma_N) = \log N\), and as \(N\) is arbitrary, we conclude that \(h_{top}(B_w) = \infty\).
	\end{proof}

	A final remark related to specification is that, in some cases---particularly for weighted shifts on \(F\)-spaces---it is possible to define an \(F\)-norm such that the operator has the specification property. This is the content of the next example.

	\begin{example}\label{ex:SP-WeightedShift}
		Let \(X\) be an \(F\)-space of sequences with an unconditional basis \({(e_n)}_n\), and \(B_w\) a backwards weighted shift on \(X\) with weight sequence \(w={(w_n)}_n\) satisfying \[ \sum_{n=1}^\infty {\bigg(\prod_{i=1}^n w_n \bigg)}^{-1} e_n \in X.\]
		Define the function
		\[ \nre{x} = \sum_{i=1}^\infty \frac{\min\{1,\nre{\pi_i x}_0\}}{2^{i}}, \textrm{ for } x \in X \]
		where \(\nre{\cdot}_0\) denotes the initial \(F\)-norm on \(X\) and \(\pi_i(x_n)_n = (x_1,\dots, x_i,0,0\dots)\).
		Now, given \(\varepsilon >0\), choose \(N \in \N\) such that \(2^{-N} < \varepsilon\). For \(x^{i} = {(x^{i}_{n})}_{n} \in X\), for \(i=1,\ldots,s\), assume we have indices 
        \[
        0 \leq a_{1} \leq b_{1} < \cdots < a_{s} \leq b_{s},\quad\textrm{ with } a_{i+1} - b_{i} \geq N. 
        \]
        Define \(z = {(z_{n})}_{n}\) by
		\[z_{n} = \left\{
		\begin{array}{ll}
			x^{i}_{n} & \mbox{if \(a_{i} \leq n < b_{i}+N\)}; \\
			0 & \mbox{otherwise}.
		\end{array} \right.\]
		Next, consider the forward shift \(F_w\) on \(X\) and define 
        \[
        \xi = \sum _{k=0}^{\infty} F_{w}^{k(b_{s}+N)} z.
        \] 
        Then \(\xi\) is \((b_{s}+N)\)-periodic and, for every \(a_{i} \leq n \leq b_{i}\), the difference \(B_{w}^{n} \xi - B_{w}^{n} x^{i}\) has the first \(n\) coordinates equal to zero. Consequently, for \(a_{i} \leq n \leq b_{i}\), we obtain
		\[
		\| B_{w}^{n} \xi - B_{w}^{n} x^{i} \| = \sum_{n=N+1}^{\infty} \frac{\min \{ 1, \| \pi _{n} ( B_{w}^{n} \xi - B_{w}^{n} x^{i}) \|_{0} \}}{2^n} 
		< \sum_{n=N+1}^{\infty} 2^{-n} = 2^{-N} < \varepsilon.
		\]
	\end{example}

	\section*{Acknowledgments}

	The authors would like to thank Alfred Peris and Mayara Braz Antunes for useful comments concerning the topics of this work.

    P.L. was partially supported by Conselho Nacional de Desenvolvimento Cient\' ifico e Tecnol\'ogico (CNPq) (grant 140036/2022{-}9).
    
    F.C.S. was partially supported by the Coordena\c{c}\~ao de Aperfei\c{c}oamento de Pessoal de N\'ivel Superior do Brasil (CAPES) (grant 682852/2022{-}00). 
    
    R.V. was partially supported by Conselho Nacional de Desenvolvimento Cient\'ifico e Tecnol\'ogico (CNPq) (grants 313947/2020{-}1 and 314978/2023{-}2), and partially supported by Funda\c{c}\~ao de Amparo \`a Pesquisa do Estado de S\~ao Paulo (FAPESP) (grants 17/06463{-}3 and 18/13481{-}0).

	\bibliography{main.bib}

\end{document}